\documentclass[12pt]{amsart}

\usepackage{amsfonts, amssymb, amscd}
\usepackage{verbatim}
\usepackage{eucal}
\usepackage{amssymb}
\usepackage{mathrsfs}
\usepackage{graphicx}
\usepackage{psfrag}
\usepackage{cite}
\usepackage{upref}

\def\Spec        {{\text{Spec }}}
\def\im           {{\text{im }}}

\def\Tor			{{\rm Tor}}
\def\Pic		{{\rm Pic}}

\def\ZZ                 {{\mathbb Z}}
\def\PP                {{\mathbb P}}

\def\CC                 {{\mathbb C}}
\def\QQ                 {{\mathbb Q}}

\newtheorem{lemma}{Lemma}[section]
\newtheorem{theorem}[lemma]{Theorem}

\newtheorem{proposition}[lemma]{Proposition}
\theoremstyle{definition}

\newtheorem{conjecture}[lemma]{Conjecture}

\newtheorem{remark}[lemma]{Remark}
\theoremstyle{remark}
\newtheorem*{proof*}{Proof}
\numberwithin{equation}{section}
\begin{document}
\title[Picard-Fuchs Equations]{A Strange Family of Calabi-Yau 3-folds}

\author{Patrick Devlin}
\address{Rutgers University, Department of Mathematics, 110 Frelinghuysen Rd.,
Piscataway \\ NJ \\ 08854 \\ USA}
\email{prd41@math.rutgers.edu}
\author{Howard J. Nuer}
\address{Rutgers University, Department of Mathematics, 110 Frelinghuysen Rd.,
Piscataway \\ NJ \\ 08854 \\ USA}
\email{hjn11@math.rutgers.edu}

\begin{abstract}
We study the predictions of mirror symmetry for the 1-parameter family of Calabi-Yau 3-folds $\tilde{X}$ with hodge numbers $h^{11}=31,h^{21}=1$ constructed in \cite{BN}.  We calculate the Picard-Fuchs differential equation associated to this family, and use it to predict the instanton numbers on the hypothetical mirror.  These exhibit a strange vanishing in odd degrees.  We also calculate the monodromy action on $H^3(\tilde{X},\QQ)$ and find that it strangely predicts a positive Euler characteristic for its mirror.  From a degenerate fiber of our family we construct a new rigid Calabi-Yau 3-fold.  In an appendix we prove the expansion of the conifold period conjectured in \cite{ES} to hold for all 1-parameter families.
\end{abstract}

\maketitle

\section{Introduction}\label{sec1}

Ever since \cite{COGP} the mathematical ramifications of mirror symmetry have revolutionized algebraic geometry and have been heartily pursued by many mathematicians.  The mirror symmetry of Calabi-Yau threefolds with one-dimensional complex moduli space, that is $h^{21}=1$, are particularly interesting since predictions about the mirror Calabi-Yau threefold are easily obtained from the so-called Picard-Fuchs equation.  Defined as a fourth-order differential equation $$(A_4(z)D_z^4+A_3(z)D_z^3+A_2(z)D_z^2+A_1(z)D_z+A_0(z))f(z)=0,$$ satisfied by the periods $$f(z)=\int_{\gamma(z)}\Omega(z)$$ associated to the one-parameter family, the Picard-Fuchs equation can be used to calculate the Yukawa coupling for the variation of Hodge structure of this one-parameter family, the monodromy action on the third cohomology group, as well as the mirror map.  Using the latter, one obtains enumerative predictions for the number of rational curves on the mirror Calabi-Yau threefold.

Studying individual examples of families of Calabi-Yau threefolds and their mirror symmetry has been a crucial step in formulating mirror symmetry as a rigorous mathematical discipline.  New examples shape and hone our understanding of the validity of techniques and definitions which have become standard in the industry by now, but are still only conjectures at the end of the day.  It is with the hope of broadening our understanding of mirror symmetry that we present here a geometric example which seems to clash at every turn with the usual prescriptions of mirror symmetry.  The authors have only seen examples with such anomalies coming from Picard-Fuchs equations constructed formally but without a geometric one-parameter family attached to them.  

We now give a brief overview of the contents of this paper.  In Section \ref{const} we provide the details of the construction of the new family of Calabi-Yau threefolds with hodge numbers (31,1) briefly mentioned in \cite{BN}.  Consider the matrices $$
\left(
\begin{array}{cccc}
a&b+c&0&0\\
b-c&a&0&0\\
0&0&0&0\\
0&0&0&0\\
\end{array}
\right),~~
\left(
\begin{array}{cccc}
0&0&0&0\\
0&a&b+c&0\\
0&b-c&a&0\\
0&0&0&0\\
\end{array}
\right),
$$
$$
\left(
\begin{array}{cccc}
0&0&0&0\\
0&0&0&0\\
0&0&a&b+c\\
0&0&b-c&a\\
\end{array}
\right),~~
\left(
\begin{array}{cccc}
a&0&0&b-c\\
0&0&0&0\\
0&0&0&0\\
b+c&0&0&a\\
\end{array}
\right), a,b,c\in \CC
$$ as sections $s_1,...,s_4$ of the bundle $Q\otimes Q$ on $G(2,4)$, the Grassmannian of planes in $\CC^4$, where $Q$ is the tautological quotient bundle.  The family of determinantal varieties in $G(2,4)$ given by the vanishing of $\det(s_1,...,s_4)$ turns out to be a one-parameter family of nodal Calabi-Yau threefolds, and in Theorem \ref{construction} we prove that they exhibit small resolutions with hodge numbers (31,1).  We also investigate the singular fibers of this family and in doing so identify a good candidate for a point of maximally unipotent monodromy (MUM point for short), as well as a new rigid Calabi-Yau threefold.

In Section \ref{PF} we recall the necessary background on variation of Hodge structure to understand the Picard-Fuchs equation, and we determine it for the one-parameter family constructed in Section \ref{const}.  We show that the candidate from Section \ref{const} is indeed an MUM point.  We also identify the existence of a second singular point which is not an MUM point in the strict sense but may nevertheless correspond to some interesting phenomena in mirror symmetry.

We use the results of Section \ref{PF} to make predictions in Section \ref{Yukawa} about the instanton numbers of the conjectured mirror family to the one constructed in this paper by computing the A-model Yukawa coupling of the mirror.  It is here that we first encounter some of the anomalies of our family as the Yukawa coupling is an even function, indicating that all odd degree Gopakumar-Vafa invariants vanish.  This is the first example known to the authors of such a strange prediction, and we know of no actual varieties exhibiting this behavior.  We offer in this section a possible explanation for this strange behavior coming from torsion in the second homology of the mirror.  We also discuss the possibility of a second mirror variety and make similar instanton number predictions for this point.  This Yukawa coupling corresponds to the one associated to equation 110 in the list of \cite{AESZ} although the Picard-Fuchs equation and monodromy are different.  We observe a very surprising relationship between these two Yukawa couplings, namely $$2\kappa_{ttt}^{\infty}(q^2)=\kappa_{ttt}^0(q),$$ where $\kappa_{ttt}^p(q)$ is the Yukawa coupling at the point $p$ considered as a power series in $q$.

Inspired by \cite{ES} we calculate the monodromy action on the middle cohomology associated to our one-parameter family, as well as the conifold period, in Section \ref{monodromy} .  In an attempt to glean information about the hypothetical mirror we used the predictions of homological mirror symmetry to translate these monodromy matrices and conifold period into information about the basic invariants of this mirror.  Specifically, we predict the size of its fundamental group, $H^3,c_2\cdot H,\text{ and }c_3$, where $H$ is the ample generator of the rank one Picard group of the mirror.  Here we find another strange facet of our family: homological mirror symmetry predicts a positive Euler characteristic of 48, which is impossible for a Calabi-Yau threefold with $h^{11}=1,h^{21}=31$.  In the end we have been unable to determine a mirror.

In an appendix, we present a proof of the form of the conifold period, $$z_2(t)=\frac{H^3}{6} t^3+\frac{c_2(Y).H}{24}t+\frac{c_3(Y)}{(2\pi i)^3}\zeta(3)+O(q),$$ which was conjectured to hold in \cite{ES} for all one-parameter families of Calabi-Yau threefolds.

{\bf Acknowledgements.}  Our debt to the papers \cite{ES} and \cite{R} cannot be overstated.  The techniques and arguments here were greatly influenced by those presented there.  The second author would like to thank his advisor Lev Borisov for his continued guidance and support.  We are also extremely happy to thank Jim Bryan, Donu Arapura, Atsushi Kanazawa, Wadim Zudilin, and Duco van Straten and his two students Michael Bogner and Joerg Hofmann for invaluable discussions and insight into different aspects of mirror symmetry.  We are also grateful to Mike Stillman and Dan Grayson for the program \textit{Macaulay2} \cite{GS} which was crucial for investigating the geometry of our family.  The second author was partially supported by NSF Grant DMS 1201466.

\section{Construction and Basic Invariants of the (31,1) family}\label{const}
We recall here the construction of a family of $(31,1)$ Calabi-Yau threefolds via the degeneration of the $(2,32)$ family of Calabi-Yau threefolds constructed in \cite{BN}.  We also fill in the details that were glossed over there and describe the basic invariants of this family.

Consider the $4\times 4$ matrices 
$$
s_1=\left(
\begin{array}{cccc}
a&b+c&0&0\\
b-c&a&0&0\\
0&0&0&0\\
0&0&0&0\\
\end{array}
\right),~~
s_2=\left(
\begin{array}{cccc}
0&0&0&0\\
0&a&b+c&0\\
0&b-c&a&0\\
0&0&0&0\\
\end{array}
\right),$$
$$
s_3=\left(
\begin{array}{cccc}
0&0&0&0\\
0&0&0&0\\
0&0&a&b+c\\
0&0&b-c&a\\
\end{array}
\right),~~
s_4=\left(
\begin{array}{cccc}
a&0&0&b-c\\
0&0&0&0\\
0&0&0&0\\
b+c&0&0&a\\
\end{array}
\right),
$$ with $a,b,c\in \CC$.  We may view these as global sections of $Q\otimes Q$, where $Q$ is the universal quotient bundle on the Grassmannian $G(2,4)$ of planes in $\CC^4$.  Consider the determinantal variety $X_{a,b,c}\subset G(2,4)$ given by the vanishing of the determinant $D:=\det(s_1, s_2,s_3,s_4)$.  For the sake of clarity we write down explicitly the equations for this variety as a complete intersection in $\PP^5$ of a quartic hypersurface with $G(2,4)$.  Pulling back $D$ and writing it in terms of the Pl\"{u}cker coordinates, which we denote by $y_0,...,y_5$, one finds this family can be written as the intersection of $G(2,4)$, defined by $h_1=y_2 y_3-y_1 y_4+y_0 y_5=0$, and the quartic hypersurface given by 
$$
\begin{aligned}
&h_2=c(a^3 y_0^2 y_1 y_2-a^3 y_0 y_1 y_2^2+2a^2by_0^2y_1y_4+a^3y_0y_1^2y_4+2a^2by_1^3y_4-\\&2a^2by_0y_1y_2y_4+4ab^2y_1^2y_2y_4+2a^2by_1y_2^2y_4+a^3y_0^2y_3y_4+2a^2by_0y_1y_3y_4+\\&4ab^2y_1^2y_3y_4+a^3y_0y_3^2y_4+2a^2by_1y_3^2y_4-a^3y_0y_1y_4^2+8b^3y_1^2y_4^2+4ab^2y_1y_2y_4^2\\&+4ab^2y_1y_3y_4^2+2a^2by_1y_4^3+4ab^2y_0^2y_1y_5-a^3y_0^2y_1y_5-4ab^2y_0y_1y_2y_5+\\&a^3y_0y_1y_2y_5-4ab^2y_0y_1y_3y_5+a^3y_0y_1y_3y_5+a^3y_1y_3^2y_5-4ab^2y_0^2y_4y_5+\\&a^3y_0^2y_4y_5-16b^3y_0y_1y_4y_5-a^3y_1^2y_4y_5-4ab^2y_0y_2y_4y_5+a^3y_0y_2y_4y_5-\\&2a^2by_1y_2y_4y_5-a^3y_2^2y_4y_5-4ab^2y_0y_3y_4y_5+a^3y_0y_3y_4y_5+2a^2by_1y_3y_4y_5+\\&a^3y_1y_4^2y_5-4ab^2y_0y_1y_5^2+a^3y_0y_1y_5^2+a^3y_1y_3y_5^2+4ab^2y_0y_4y_5^2-a^3y_4y_5^2+\\&2a^2by_1y_4y_5^2+a^3y_2y_4y_5^2)=0.
\end{aligned}
$$
Notice from the equations that we may assume $c\neq 0$ and scale it to be 1.  Moreover, scaling the two parameters $a,b$ simultaneously doesn't change the variety, so we may view this family as being over $\PP^1$.  

We have the following result enumerating the important properties of $X:=X_{a,b}$:
\begin{theorem}\label{construction} For generic choices of $a,b$ we have

(1) $X$ is an irreducible threefold whose singular locus consists of 118 ordinary double points.

(2) There is a small resolution $\pi:\tilde{X}\rightarrow X$ 
of the ordinary double points, with $\tilde{X}$ a
non-singular Calabi-Yau threefold.

(3) $\chi(\tilde{X})=60$, 
$h^{1,1}(\tilde{X})=31$, and
$h^{2,1}(\tilde{X})=1$. \end{theorem}

\begin{proof} (1) We note that $X$ is the complete intersection defined by $h_1,h_2$ and thus is certainly connected.  Irreducibility will follow from the fact that the singular locus is zero dimensional.  To see that its singular
locus is as claimed for generic choice of parameters, we may check on each standard affine open subset $U_i$ of $G(2,4)$, given by $\{y_i\neq 0\}$, which is isomorphic to $\CC^4=\Spec \CC[x_1,...,x_4]$.  $X\cap U_i$ then becomes a hypersurface, say $V(h)$, and the locus of worse-than-nodal points can be described by adding the Hessian of $h$ to the Jacobian ideal of $X\cap U_i$.  By eliminating the coordinate variables from this ideal for each $i$, for example by using \textit{Macaulay2}, one finds that the worse-than-nodal locus on $\PP^1$ is given by $a^2(a^2-b^2)b^2=0$.  By checking over all parameter values in $\PP^1$ over various finite fields one easily sees that 118 is the generic number of nodes.  (An exceptional parameter value for which $X$ has more than 118 nodes will be discussed later in the paper).

(2) We construct one such small resolution using \textit{Macaulay2}.  First notice that we can scale $b$ to be 1 since $\infty\in \PP^1$ has unnodal fiber anyway.  Then we consider the intersection of $Sing(X)$ with the union of the coordinate hyperplanes given by $V(y_0\cdot\cdot\cdot y_5)$.  By calculating the Hilbert polynomial of $Sing(X)\cap V(y_0\cdot\cdot\cdot y_5)$, one finds that 114 of the singular points are contained in this locus.  We consider the effective Cartier divisor defined by $y_0\cdot\cdot\cdot y_5$ on $X$.  Using \textit{Macaulay2}, we can explicitly calculate the irreducible components $\{S_i\}_{i=0,...,18}$ of this Cartier divisor and find that 13 of them, say $S_0,...,S_{12}$, are smooth surfaces.  Moreover, all 114 of the above-mentioned singular points lie on the union of these 13 surfaces, and thus these are non-Cartier Weil divisors.  If we blow up $X$ along $S_0$, then we obtain a small partial resolution $X^0$ of the nodes of $X$ that lie on $S_0$.  By considering the proper transform $\tilde{S}_1$ of $S_1$ on $X^0$, we proceed to blow up $X^0$ along $\tilde{S}_1$.  Above the nodes we resolved previously $X^0$ is smooth and thus $\tilde{S}_1$ is Cartier there, so this blow-up has no effect there, and we obtain another small partial resolution $X^1$ of the nodes of $X$ contained in $S_0\bigcup S_1$.  Proceeding in this way, we find a small partial resolution $\hat{X}$ of 114 of the 118 nodes of $X$.  For the remaining 4 nodes, one notices that they lie on the hyperplane given by $V(y_1-y_4-ay_5+ay_0)$.  Again consider the Cartier divisor on $X$ given by $V(y_1-y_4-ay_5+ay_0)\cap X$.  It has 4 irreducible components which are all smooth surfaces.  By blowing each of these up in succession as before, we resolve the remaining 4 nodes to obtain a small resolution $\tilde{X}$.  That it is a Calabi-Yau is standard, but details are given in Section 7 of \cite{BN}.

(3) The claim about the topological Euler characteristic follows immediately from the
fact that a generic (2,4) complete intersection has Euler characteristic -176, and from
the number of nodes.  The calculation of the Hodge numbers follows from the results of Section 7 in \cite{BN} and \textit{Macaulay2} calculations.  
\end{proof}

\begin{remark} It is worth noting that the small resolution above is not canonical.  The order in which the $S_i$ are blown up likely determines different smooth birational models of $X$ related by flops.  To obtain a simultaneous small resolution of the family over the locus of $\PP^1$ parametrizing fibers with 118 nodes, one must take an \'{e}tale cover of this locus to eliminate the possible monodromy amongst the $S_i$.  Since the B-model is unaffected by resolution of singularities, we may work on the explicit singular family above and use it as a universal family of objects over the moduli space.  This will allow us to use mirror symmetry below to make predictions about the hypothetical mirror.
\end{remark}

\subsection{Degenerate Fibers}

Using \textit{Macaulay2} one can verify that this family becomes reducible for $(a,b)=(0,1),(\pm 1,1),(1,0)$.  Over $(0,1)$ it becomes the union of a degree 4 complete intersection with 6 nodes given by $2y_2y_3-y_1y_4=2y_0y_5-y_1y_4=0$, and two quadric cones given by $y_4=y_2y_3+y_0y_5=0$ and $y_1=y_2y_3+y_0y_5=0$, respectively.  The decomposition over $(1,1)$ is harder to see explicitly, but the Hilbert polynomial of its singular locus has degree 2, so it must certainly be reducible there.  At infinity, the family becomes the union of a degree 2 component given by $y_0-y_2+y_3+y_5=-y_2y_3+y_1y_4-y_2y_5+y_3y_5+y_5^2=0$, which is a quadratic cone, and a degree 6 component given by $y_2y_3-y_1y_4+y_0y_5=y_0y_1y_2+y_0y_3y_4+y_1y_3y_5+y_2y_4y_5=0$ with 34 nodes.  

From the proof of the above theorem, we saw that the nodal singularities of the generic fibre live along the union of the coordinate axes with another moving hyperplane given by $y_1-y_4-ay_5+ay_0$.  One can use elimination theory again to determine if there are any fibers which acquire additional singularities (this time by localizing at $y_0\cdot\cdot\cdot y_5\cdot (y_1-y_4-ay_5+ay_0)$ to avoid the known singularities).  We indeed find that $(\pm\sqrt{-8},1)$ gives such a fiber.  By checking over various finite fields fiber by fiber, we became quite certain that there are no others.  We'll see below from the Picard-Fuchs equation  that this is indeed the case.

\subsection{A New Rigid Calabi-Yau}\label{rigid}
The fiber of our family above $a=\sqrt{-8}$ is a nodal Calabi-Yau threefold with 122 nodes.  One can easily verify that it has a crepant resolution, obtained similarly as before, and using the techniques of Section 7 of \cite{BN} we found that the hodge numbers of the resolution are $(34,0)$.  To the best of our knowledge this is a new example of a rigid Calabi-Yau threefold.  As traditional mirror symmetry does not apply to this variety, it would be interesting to understand this example in the context of generalized mirror symmetry as in \cite{CDP}.

\subsection{Effectivity of this family}
Although it is not obvious from looking at the equations, the above family is slightly ineffective in a moduli-theoretic sense since $X_{a,1}\cong X_{-a,1}$:

\begin{proposition}\label{eff}  There is a natural isomorphism $X_{a,1}\cong X_{-a,1}$
\end{proposition}
\begin{proof}  Consider the linear map on $\CC^4$ with basis $e_1,...,e_4$ given by $e_1\mapsto -e_1,e_2\mapsto e_2,e_3\mapsto -e_3,e_4\mapsto e_4$. Then this has the effect on the matrices $s_i$ of negating the entries which involve $b,c$ while leaving the rest of the matrix unchanged.  Thus $X_{a,b,c}\cong X_{a,-b,-c}$ in the original notation.  But negating the sign of all entries of each matrix $s_i$ doesn't change the determinantal locus at all, so $X_{a,-b,-c}=X_{-a,b,c}$.  Thus indeed $X_{a,1}\cong X_{-a,1}$.   
\end{proof}

This tells us that $a^2$, not $a$ is the natural moduli parameter for this family.  We will switch to this parameter shortly.

\subsection{Integral homology of $\tilde{X}$}

For the sake of completeness and for later use, we record here the important facts about the integral homology of our $\tilde{X}$:

\begin{proposition}\label{top} The nonsingular small resolution $\tilde{X}$ of the complete intersection $X$ is simply connected.  Moreover, for its homology we have 

(i) $H_i(\tilde{X})=H_i(\PP^3)$ for $i\neq 2,3,4$;

(ii) $H_4(\tilde{X})$ is torsion free;

(iii) $\Tor(H_3(\tilde{X}))=\Tor(H_3(X))$.
\end{proposition}
\begin{proof} That the complete intersection $X$ is simply connected follows from Corollary 5.2.4 in \cite{Dim}.  The resolution $\pi:\tilde{X}\rightarrow X$ replaces the nodes by simply connected $\PP^1$'s, so it too is simply connected.  For the statement about homology, we first note that by Theorem 5.4.3 and Corollary 5.4.4 in \cite{Dim} we have $H_j(X)=H_j(\PP^3)$ for $j\neq 3,4$, and $H_4(X)$ is torsion-free.  Now consider the union $U$ of small open balls around the nodes of $X$ and its preimage $V=\pi^{-1}(U)$.  Then by excision for the union of the nodes $W$ (the exceptional locus $E$ of $\pi$, respectively) we get that $H_i(X-W,U-W)\cong H_i(X,U)$ (respectively, $H_n(\tilde{X}-E,V-E)\cong H_i(\tilde{X},V)$).  Thus clearly $H_n(X,U)\cong H_n(\tilde{X},V)$ since the excised relative homologies are obviously isomorphic.  Using the long exact sequence of homology for the pair $(\tilde{X},V)$ and the fact that $V$ deformation retracts onto $E$, topologically a union of $S^2$'s, we get that $H_i(\tilde{X})\cong H_i(\tilde{X},V)$ for $i>3$.  We also obtain two exact sequences: $$0\rightarrow H_1(\tilde{X})\rightarrow H_1(\tilde{X},V)\rightarrow \ZZ^{117}\rightarrow 0, \text{and}$$ $$0\rightarrow H_3(\tilde{X})\rightarrow H_3(\tilde{X},V)\rightarrow \ZZ^{118}\rightarrow H_2(\tilde{X})\rightarrow H_2(\tilde{X},V)\rightarrow 0.$$  Doing the same for the pair $(X,U)$ and using the fact that $U$ deformation retracts onto 118 points, we see that $H_i(X)\cong H_i(X,U)$ for $i>1$ and we get an exact sequence $$0\rightarrow H_1(X)\rightarrow H_1(X,U)\rightarrow \ZZ^{117}\rightarrow 0.$$  It follows from this that $H_i(X)\cong H_i(\tilde{X})$ for $i\neq 2,3$ which proves \textit{(i)} and \textit{(ii)}.  Since $H_2(\tilde{X},V)\cong H_2(X,U)\cong H_2(X)\cong \ZZ$, we get that $H_2(\tilde{X})\cong \ZZ\oplus M$ and $H_3(X)\cong H_3(\tilde{X})\oplus K$, where $M=\im(\ZZ^{118}\rightarrow H_2(\tilde{X}))$ and $K=\ker(\ZZ^{118}\rightarrow H_2(\tilde{X}))$.  Since $K$ is free, \textit{(iii)} follows as well.  
\end{proof}
\section{Picard-Fuchs Equation and Maximally Unipotent Monodromy}\label{PF}
\subsection{Theory behind the approach}

For details and proofs of the techniques used in this section, see \cite{CK}.  For notational simplicity, we assume $b=1$ and write our family as $\tilde{X}_a$ with moduli parameter $a$.  

Let $S=\PP^1-\{0,\pm 1,\pm \sqrt{-8},\infty\}$ be the uncompactified complex moduli space of our family, and let $\mathcal F^0=R^3\pi_* \CC\otimes \mathcal O_S$ be the induced local system, with subbundles $\mathcal F^p$ which correspond to the Hodge filtration on fibers.  By the nilpotent orbit theorem of Schmid in \cite{Sch}, we may canonically extend these bundles to a filtration of bundles $$0\subset \overline{\mathcal F}^3\subset ...\subset \overline{\mathcal F}^0,$$ on the compactification, $\PP^1$.  Since $\overline{\mathcal F}^3$ is a line bundle (as $\pi$ is a family of Calabi-Yau 3-folds and thus $h^{3,0}=1$), we may choose a fixed local generator $\Omega \in \overline{\mathcal F}^3$ around $a=0$.  The Picard-Fuchs equation is a differential equation satisfied by the periods $f(a)=\int_{\gamma(a)} \Omega(a)$ for $\gamma(a)$ a 3-cycle on $\tilde{X}_a$ moving continuously with $a$.

To see where the Picard-Fuchs equation comes from, we note that if $\nabla$ is the Gauss-Manin connection on $\mathcal F^0$ (we also use the same notation for the logarmithmic extension of it to the compactification, due to \cite{Del1}), then according to \cite{BG} $\Omega,\nabla_{\delta} \Omega,\nabla_{\delta}^2 \Omega,\nabla_{\delta}^3\Omega$ are generically linearly independent, where $\delta=a\frac{d}{da}$ is seen as a tangent vector along $S$, and thus form a basis in a punctured neighborhood of the origin.  Applying $\nabla_{\delta}$ once more then gives a relation of the form $$\nabla_{\delta}^4\Omega+B_3(a) \nabla_{\delta}^3\Omega+B_2(a)\nabla_{\delta}^2\Omega+B_1(a) \nabla_{\delta}\Omega+B_0(a) \Omega=0,$$ where the $B_i$ are holomorphic near 0 because Deligne's extension has only regular singular points.  We may of course clear denominators to obtain a holomorphic coefficient in front of $\nabla_{\delta}^4 \Omega$, and by restricting to algebraic differentials the coefficients must then be polynomial.  Applying this operator to any period $f(a)=\int_{\gamma(a)}\Omega$ shows that it must satisfy a differential equation, $$\delta^4 f(a)+B_3(a)\delta^3 f(a)+B_2(a)\delta^2 f(a)+B_1(a)\delta f(a)+B_0(a) f(a)=0,$$called the Picard-Fuchs equation.  

We must also consider the monodromy action $\mathcal T:H^3(\tilde{X},\CC)\rightarrow H^3(\tilde{X},\CC)$ obtained by going around the point $a=0$, where $\tilde{X}$ is a smooth fiber over a point very close to the origin.  According to the monodromy theorem \cite{Lan}, this linear action is quasi-unipotent, i.e there exist positive integers $n,m$ such that $$(\mathcal T^n-1)^m=0, (\mathcal T^n-1)^k\neq 0 \text{ for }k<m,$$ with unipotent index $m$ at most 4, since $\dim H^3(X,\CC)=2h^{3,0}+2h^{2,1}=4$.  An MUM point is defined by the condition that $n=1,m=4$.  Thus we can find a basis $g_0,g_1,g_2,g_3$ of $H^3(\tilde{X},\CC)$ such that $$\mathcal T=\left(\begin{matrix} 1 &1&0&0\\0&1&1&0\\0&0&1&1\\0&0&0&1\end{matrix}\right).$$  Then we can choose a basis in homology $\gamma_0,\gamma_1,\gamma_2,\gamma_3$ Poincare dual to this basis.  

 According to the nilpotent orbit theorem, $g_0$ extends to a single-valued flat section of $\overline{\mathcal F}^0$, and accordingly $f_0=\int_{\gamma_0}\Omega$ extends to a single-valued holomorphic function at $a=0$.  If follows from the form of $\mathcal T$ that analytically continuing $f_i=\int_{\gamma_i}\Omega$ around $a=0$ then gives $f_i+f_{i-1}$.  It follows that at an MUM point there should be two cycles $\gamma_0,\gamma_1$ whose periods, $f_0,f_1$ respectively, satisfy $t:=f_1/f_0=g+\log a$, where $g$ is holomorphic at 0.

  We have the following proposition from \cite{CK} that allows us to use the Picard-Fuchs equation to determine the type of monodromy we have around a given boundary point:

\begin{proposition}\label{CK}  Let the Picard-Fuchs equation for a given local section $\Omega$ of $\overline{\mathcal F}^0$ be written as $$\delta^4 f+B_3(a)\delta^3 f+B_2(a)\delta^2 f+B_1(a) \delta f+B_0(a) f=0.$$ Then the monodromy action $\mathcal T$ is unipotent if and only if the roots of the indicial equation are integers.  Furthermore, $\mathcal T$ is maximally unipotent if and only if the indical equation is of the form $(y-l)^4=0$ for some integer $l$.
\end{proposition}

One can check that this integer $l$ is insignificant in that altering the given holomorphic 3-form $\Omega$ by multiplying by a meromorphic function $p(a)=c_l a^l+...$ with $c_l\neq 0$ alters the indicial equation of the new differential equation by replacing $y$ with $y-l$.  So the significant part is that all roots of the indicial equation be equal.

\subsection{Determining the Picard-Fuchs equation near $a=0$}

We saw above that the fiber above $a=0$ was highly degenerate, breaking into three irreducible components.  This suggests that $a=0$ is a good candidate for an MUM point.  Since the $B$-model (that is the Hodge theoretic side of mirror symmetry) is unchanged via desingularization, we may do all of our calculations on the complete intersection $X_a$.  We may also restrict ourselves to the open affine $U_0\subset \PP^5$ given by $y_0\neq 0$ and choose a 3-cycle   inside the open subset $X_a\cap U_0$ which various continuously with $a$ and against which we will integrate a local section of $\overline{\mathcal F}^3$ to obtain a period.  

We may further simplify the calculation by considering $X_a\cap U_0$ as the hypersurface in $\CC^4$ with coordinates $y_1,...,y_4$ defined by the equation $h(y_1,...,y_4)=h_2(1,y_1,...,y_4,y_1 y_4-y_2 y_3)$.  The natural choice for a local section of the canonical bundle of $X_a$ on $U_0$ is then given by the residue of \[\Psi=(\frac{1}{2\pi i})^4 \frac{dy_1\wedge ...\wedge dy_4}{h}.\]  We may choose a constant 4-cycle $\Gamma$ on $\CC^4\backslash X_a\cap U_0$, and then a period can be obtained as $$f_0(a)=\int_{\Gamma} \Psi(a).$$  Since all periods must satisfy the Picard-Fuchs equation, the choice of such a 4-cycle is irrelevant for our purposes.  One obviously must worry about holomorphicity, but we can choose an appropriate 4-cycle in a moment. 

The essence of the approach we take follows that of R{\o}dland in \cite{R}.  We find that $$\Psi=\frac{1}{16}\cdot\frac{1}{ 1-\sum_{i} v_i}\cdot \bigwedge \frac{dy_i}{(2\pi i) y_i},$$ where this comes from writing $\frac{y_1 \cdot\cdot\cdot y_4}{h}$ as \[\frac{1}{(\frac{h}{y_1\cdot\cdot\cdot y_4})},\] and writing the denominator in terms of the Laurent monomials $v_i$.  Then we take the 4-cycle $\Gamma$ to be a topological torus given $|y_i|=\epsilon_i$ and consider the geometric series expansion of $$\frac{1}{ 1-\sum_{i} v_i}.$$ For this series to converge and to allow us to manipulate it freely, it suffices to ensure that $\sum |v_i|<1$.  We can do this by choosing $\epsilon_1=\epsilon_3=.05,\epsilon_2=\epsilon_4=.5$, and $|a|<1/40$ for example.  When expanding the resulting absolutely convergent series as a Laurent series in terms of the $v_i$, the only terms that contribute to the integral are those without any $y_i$'s.  Indeed, if a Laurent monomial has either negative powers or positive powers for some $y_i$, then when multiplied by $\frac{dy_i}{(2\pi i) y_i}$ the resulting function of $y_i$ being integrated either has a removable singularity (the total exponent would then be at most -2) or is holomorphic, and thus the integral of this function over the closed circle $|y_i|=\epsilon_i$ is 0.  

Now generators of the subring of $\CC[\{v_i\}]$ generated by those monomials that are independent of the $y_i$'s can be calculated as in the appendix to \cite{R}, and a closed form for the period $f_0(a)$ can be obtained as R{\o}dland does.  Unfortunately, because of the number of Laurent monomials $v_i$ and the fact that the generators of this subring have different powers of $a$, the resulting form of the period is almost useless and involves an infinite sum in approximately 40 indices.  Instead we may calculate iterated residues in \textit{Maple} to compute the power series expansion of $f_0$ to a large number of terms.  We find that $$f_0(a)=-\frac{1}{16}-\frac{3}{64}a^2-\frac{81}{2048}a^4-\frac{143}{4096}a^6-\frac{66357}{2097152}a^8-...,$$ where we notice that this function is entirely even.  This is what we expect from Proposition \ref{eff}.

Converting to the true moduli parameter $z=a^2$, we use recurrence relations on the coefficients of the power series to determine what polynomials $A_i(z)$ satisfy the Picard-Fuchs equation $$A_4(z)D_z^4 f_0(z)+...+A_1(z)D_z f_0(z)+A_0(z) f_0(z)=0,$$ where $D_z=z\frac{d}{dz}$.  We find a solution in degree 6 given by the operator  $$\begin{aligned}\mathcal D:=&16(z-1)^3(z-4)^2(z+8)D_z^4+96z(z-1)^2(z-4)(z^2-28)D_z^3\\&+12z(z-1)(18z^4-129z^3-136z^2+2000z-1024)D_z^2\\&+36z(192-752z+540z^2+99z^3-58z^4+6z^5)D_z\\&+3z(512-2688z+1824z^2+856z^3-288z^4+27z^5).\end{aligned}$$  We notice from the leading coefficient that as expected the only singular fibers are at $z=0,1,-8, \infty \in \PP^1$.  One can check that $z=4$ does not represent an actual singular fiber, and we'll see that it is the unique vanishing point for the B-model Yukawa coupling of this family.  From calculating the indicial equation of this ODE, we immediately get the following proposition:

\begin{proposition}  The point $z=0$ is an MUM point of the family of Calabi-Yau 3-folds constructed here.
\end{proposition}

\subsection{The Picard-Fuchs equation at other singular fibers}

In the new coordinate $z$, we find that there are 3 other singular values to check, $z=1,-8,\infty$.  Translating our ODE accordingly around those points, one finds that the Riemann scheme for our differential operator is
$$P\begin{Bmatrix}
\begin{tabular}{c c c c}
-8 & 0 & 1 & $\infty$ \\ 
\hline
0&0&0&3/2 \\
1&0&0&3/2 \\
1&0&-1/2&3/2 \\
2&0&1/2&3/2 \\
\end{tabular}
\end{Bmatrix}.$$  From Proposition \ref{CK} we see that the monodromy is only unipotent outside the origin around the point $z=-8$.  This value corresponds to where the family acquires extra nodes and is known as a conifold point.  That $z=-8$ has spectrum (this is the set of roots of its indicial equation) $\{0,1,1,2\}$ agrees with predictions made in \cite{ES} about the kind of monodromy around such a singular fiber and its spectrum.  

The spectrum for $\infty$ is somewhat strange.  It is not an MUM point in the strict sense, but if one takes a double cover of $\PP^1$, for example branched at 1 and $\infty$, then the spectrum  at $\infty$ becomes that of a genuine MUM point.  At first glance this suggests the existence of a second mirror family to $\tilde{X}_z$.  According to homological mirror symmetry this family should conjecturally be derived equivalent to the mirror family corresponding to the MUM point at $z=0$.  The validity of $z=\infty$ as an MUM point will be elaborated on later.

\section{Mirror Symmetry and A-model Yukawa couplings}
According to mirror symmetry, the B-model (Hodge-theoretic) Yukawa coupling on a Calabi-Yau threefold $X$ gives the A-model (Gromov-Witten) Yukawa coupling on its mirror $Y$ after applying the mirror map.  

The A-model Yukawa coupling is defined in terms of the genus zero Gromov-Witten (GW) invariants, and can be seen (in the Picard rank one case) as the generating function of the genus zero Gromov-Witten invariants over all degrees.  More specifically, the A-model Yukawa coupling is written as $$\kappa_{ttt}=H^3+\sum_{d=1}^{\infty} N_d q^d,$$ where $H$ is the ample generator of $\text{Pic}(Y)$, $N_d:=\int_{[\overline{M}_{0,0}(Y,d)]^{virt}}1$ is the degree $d$ unpointed genus 0 Gromov-Witten invariant, defined by integrating over the virtual fundamental class $[\overline{M}_{0,0}(Y,d)]^{virt}$ of the moduli space of stable maps $\overline{M}_{0,0}(Y,d)$.  Conjecturally, these are integers, and this power series can further be written as $$\kappa_{ttt}=H^3+\sum_{d=1}^{\infty} n_d \frac{d^3 q^d}{1-q^d},$$ where the $n_d$ are the Gopakumar-Vafa (GV) invariants, or instanton numbers, which naively should count the number of rational curves on $Y$ of degree $d$(degree being measured against the ample generator of Pic$(Y)$).  These numbers are also conjectured to be integral.  The relation between these two invariants is given by $$N_d=\sum_{k|d} n_{\frac{d}{k}} k^{-3}.$$  It should be noted that this relation is often the definition of the GV invariants since a rigorous mathematical definition of them does not exist at the moment.  The enumerative significance of both the GW and GV invariants is a subtle issue.

If we take $f_0$ to be the unique holomorphic solution to the Picard-Fuchs equation around an MUM point, and choose a period $f_1$ such that $t=f_1/f_0=g+\log z$ with $g$ holomorphic at $z=0$ and normalized so that $g(0)=0$, the mirror map is defined by $q=e^t$.  Obviously $q$ is determined only up to a constant.  Mirror symmetry then predicts that the A-model Yukawa coupling defined above is in fact equal to the B-model Yukawa coupling, that is $$\kappa_{ttt}=(\frac{d\log z}{dt})^3\frac{1}{f_0(z)^2}\int_{X_z} \Omega\wedge \nabla_{z\frac{d}{dz}}^3 \Omega,$$ where the right hand side is the B-model Yukawa coupling calculated by Hodge theory on the mirror family.

\subsection{A-model Yukawa coupling prediction}\label{Yukawa}

Applying these ideas to our family, we write $g$ as a power series in $z$ and substite the resulting formula for $f_1$ back into the Picard-Fuchs equation to get $$g(z)=\frac{3}{8}z+\frac{81}{512}z^2+\frac{187}{2048}z^3+\frac{64797}{1048576}z^4+..., \text{ and}$$$$q=c_2 z(e^g)=c_2(z+\frac{3}{8}z^2+\frac{117}{512}z^3+\frac{653}{4096}z^4+...).$$  

Using the Picard-Fuchs equation and Griffith's transversality, it can be shown that $L=\int_{X_z} \Omega\wedge \nabla_{z\frac{d}{dz}}^3 \Omega$ satisfies the differential equation $$\frac{dL}{L}=-\frac{A_3(z)}{2zA_4(z)},$$ so that in our case $$L=c_1 \frac{z-4}{(z-1)^3(z+8)}.$$ Here we notice that $z=4$ is the unique vanishing point of $L$.  

Finally, one can calculate the inverse series giving $z=z(q)$ as a power series in $q$, and using that $\frac{d \log z}{dt}=\frac{q}{z}\frac{dz}{dq}$, we put everything together to find that according to mirror symmetry $$\begin{aligned}\kappa_{ttt}^0&=m(2+27q^2-2232q^4+43617q^6-7425720q^8+...)\\&=m(2+9\frac{2^3 q^2}{1-q^2}-36\frac{4^3q^4}{1-q^4}+2019\frac{6^3 q^6}{1-q^6}+...),
\end{aligned}$$ where we've found the unique choice $c_2=1/32$ which makes all of the $n_d$ integral.  We've also let $c_1=4m$ to scale away the rest of the denominators.  This leads to the 
\begin{conjecture} The genus 0 GV invariants $n_d$ of the mirror to $\tilde{X}$ are given by the above values. 
\end{conjecture}

Obviously this Yukawa coupling is strange since it suggests that $n_d=0$ for all odd $d$.  We discuss a possible explanation in the next section.  It is worth noting that the conjectured integrality of GW/GV invariants holds for our family, and moreover our family provides the first example known to the authors (and others) of such an even Yukawa coupling coming from geometry.

\subsection{Torsion in homology and the vanishing of odd GV invariants}

It is explained in \cite{AM} that if torsion is present in $H_2(Y,\ZZ)$, then the $A$-model Yukawa coupling takes on a more complicated form.  For example, if $H_2(Y,\ZZ)\cong \ZZ\times \ZZ_2$, then we find that $$\kappa_{ttt}=H^3+(n_1^0+(-1)^a n_1^1)q+O(q^2),$$ where $H^3=\deg Y$, $a=0$ or 1, and $n_1^i$ are the number of lines whose torsion component in homology is in the class $i\in\ZZ_2$.  Then $n_1^0+n_1^1$ is the total number of lines.  In this case if we had $n_1^0=n_1^1$ and $a=1$, then the corresponding coefficient of $q$ would vanish.  Assuming this paradigm continues in higher degrees, this could be a possible explanation of the vanishing of the odd coefficients in our Yukawa coupling.

It is worth noting that the predictions in \cite{BK} cannot apply in total generality to all families of Calabi-Yau threefolds, as the examples of \cite{HT},\cite{S} of double mirrors with different fundamental groups show.  Therefore torsion in homology may indeed explain the vanishing of the odd GW/GV invariants above.  Moreover, as noted in \cite{BK} $\Tor(H_2(Y,\ZZ))$ is dual to $Br(Y)$, the Brauer group of $Y$.  Thus nontrivial torsion in the second homology would correspond to a nontrivial Brauer group as well (see \cite{A}, \cite{HT} for examples of Calabi-Yau threefolds with nontrivial Brauer group).  We note the possible relevance of this later.

\subsection{Virtual GW/GV invariants at infinity}
As we observed above, $\infty$ is not a classicaly defined MUM point, but it almost is.  This type of point appears also in the families constructed by Kanazawa in \cite{Kan}.  As in those examples, we may transform our operator to the point at infinity and calculate the virtual GW/GV invariants that are conjectured to hold on the virtual mirror.  We use the term \textit{virtual} since it's not clear that mirror symmetry predicts the existence of a mirror or the equality of Yukawa couplings in this case.  Nevertheless, changing the coordinate from $z$ to $1/z$ and transforming the gauge by $\sqrt{z^3}$ transforms the Euler operator by $D_z\rightarrow -D_z-3/2$.  Thus the Picard-Fuchs operator at infinity becomes 
$$\begin{aligned}\tilde{\mathcal D}=&16(z-1)^3(4z-1)^2(8z+1)D_z^4\\&+96z(z-1)^2(4z-1)(32z^2-8z+3)D_z^3\\&+12z(z-1)(2304z^4-2464z^3+992z^2-118z+15)D_z^2\\&+36z(4z-1)(192z^4-304z^3+150z^2-30z+1)D_z\\&+3z^2(3456z^4-5840z^3+3408z^2-933z+152).\end{aligned}$$  Now this operator has an MUM point at the new origin, so we may proceed as before, and assuming the equality of the B-model Yukawa coupling with the virtual A-model Yukawa coupling, we get that after choosing the constants of integration to be $c_1=m$ and $c_2=-2^{-4}$ the Yukawa coupling is $$\begin{aligned}\kappa_{ttt}^{\infty}&=m(1+36q-1116q^2+218088q^3-3712860q^4+...)\\&=m(1+36\frac{q}{1-q}-144\frac{2^3q^2}{1-q^2}+8076\frac{3^3q^3}{1-q^3}-57996\frac{4^3q^4}{1-q^4}+...).
\end{aligned}$$  It is interesting to see that this Yukawa coupling again satisfies the integrality conjecture expected of genuine GW/GV invariants.  The mathematical meaning of these virtual numbers is not at the moment understood.  It would be important to understand them better in the future.  M. Bogner has pointed out to the second author that this Yukawa coupling is the same as that of equation 110 in the list of \cite{AESZ}.  The differential operator itself and the associated monodromy are different however.  

Upon reviewing the sequence of instanton numbers in $\kappa_{ttt}^{\infty}$ and comparing with those of $\kappa_{ttt}^0$, it becomes clear that $n^{\infty}_{d}=4n_{2d}^{0}$.  In terms of power series expansions one in fact has $2\kappa_{ttt}^{\infty}(q^2)=\kappa^0_{ttt}(q)$, at least up to 100 terms.  The meaning of these relations, geometric or otherwise, is unknown at the moment, but it suggests an intimate connection between the behavior at 0 and $\infty$.

\section{Calabi-Yau differential equations, Monodromy, and the search for a mirror pair}\label{monodromy}

As mentioned in the introduction, we have been unsuccesful in determining a mirror candidate to $\tilde{X}$.  In an attempt to determine a possible mirror, we used the observation that the Picard-Fuchs equation of a one-parameter family can be used to calculate the monodromy action of the family, and these matrices can be conjecturally used to determine the fundamental numerical invariants of its mirror.  In addition to the strange Yukawa coupling above, we find that our family also presents strange results in the context of these by-now standard techniques in mirror symmetry.  To see this, we follow the approach in \cite{ES} for calculating the monodromy matrices associated to a \textit{Calabi-Yau differential equation}.  

\subsection{Calabi-Yau Differential Equations}

These equations were defined in \cite{AESZ} as fourth-order ODE's $$\frac{d^4f}{dz^4}+a_3(z)\frac{d^3 f}{dz^3}+a_2(z)\frac{d^2 f}{dz^2}+a_1(z)\frac{df}{dz}+a_0(z) f(z)=0,$$ satisfying the following five conditions which are expected to correspond to equations obtained as Picard-Fuchs equations of genuine one-parameter families:

(1) The singular point at $z=0$ is an MUM point;

(2) The coefficients $a_i(z)$ satisfy the equation $$a_1=\frac{1}{2}a_2 a_3-\frac{1}{8}a_3^3+a_2'-\frac{3}{4}a_3 a_3'-\frac{1}{2}a_3'';$$

(3) The solutions $\lambda_1\leq \lambda_2\leq\lambda_3\leq\lambda_4$ of the indicial equation at $z=\infty$ are positive rational numbers satisfying $\lambda_1+\lambda_4=\lambda_2+\lambda_3,$ and we suppose that the eigenvalues of the monodromy around $z=\infty$ are the zeros of a product of cyclotomoic polynomials;

(4) The power series solution near $z=0$ has integral coefficients;

(5) The Yukawa coupling satisfies the integrality conjecture up to multiplication by a positive integer.  

One can easily check that our Picard-Fuchs equations at 0 and $\infty$ satisfy all of the above condition except (4), but this pathology is not serious.  By scaling the moduli parameters by 32 and -16, respectively, we obtain new ODE's which now satisfy all of the above conditions.  Of course, this does not actually change the family or the Yukawa couplings calculated above.

\subsection{Monodromy for the operator $\mathcal D$}
The monodromy matrices for the action of $\pi_1(S,p)$ on $H^3(\tilde{X},\QQ)$ associated to our family $\tilde{X_z}$ can be calculated numerically using \textit{Maple} as in \cite{ES} ($S=\PP^1-\{0,1,-8,\infty\}$ as above).  We obtain the monodromy matrices to be 
$$
T_0=\left(
\begin{array}{cccc}
1&0&0&0\\
1&1&0&0\\
\frac{1}{2}&1&1&0\\
\frac{1}{6}&\frac{1}{2}&1&1\\
\end{array}
\right),~~
T_{-8}=\left(
\begin{array}{cccc}
1&-8&0&-192\\
0&1&0&0\\
0&0&1&0\\
0&0&0&1\\
\end{array}
\right),~~T_4=Id,
$$
$$
T_1=\left(
\begin{array}{cccc}
8&-28&72&-96\\
\frac{7}{2}&-11&24&-24\\
\frac{5}{8}&-\frac{3}{2}&2&0\\
\frac{1}{48}&\frac{1}{12}&-\frac{1}{2}&1\\
\end{array}
\right),~~
T_{\infty}=\left(
\begin{array}{cccc}
-8&28&-120&96\\
-\frac{1}{2}&1&0&-24\\
\frac{1}{8}&-\frac{1}{2}&2&0\\
-\frac{1}{48}&\frac{1}{12}&-\frac{1}{2}&1\\
\end{array}
\right),
$$ where the matrices are presented in the first standard form mentioned in \cite{ES}.  It is certainly comforting that the monodromy around $z=4$ calculated in this way yields the identity, as one would expect from a smooth fiber.  

According to homological mirror symmetry (HMS), the relationship between the Calabi-Yau $\tilde{X}$ and its conjectural mirror $Y$ goes beyond just the mirror duality of the hodge diamond or the equality of the A and B-model Yukawa couplings.  Kontsevich expressed this symmetry by the stronger equivalence of two different derived categories associated to Calabi-Yau threefolds.  The main statement of HMS is that the bounded derived category of coherent sheaves on $\tilde{X}$, $D^b(\tilde{X})$, is equivalent to the derived Fukaya category of $Y$, $D\mathcal F(Y)$, and vice-versa.  Via the chern character this descends to cohomology as $$H^{ev}(\tilde{X},\QQ)\cong H^3(Y,\QQ),H^{ev}(Y,\QQ)\cong H^3(\tilde{X},\QQ).$$  Accordingly , the monodromy action on $H^3(\tilde{X},\QQ)$ above should correspond to autoequivalences of $D^b(Y)$ which descend to automorphisms of the even cohomology.  Moreover, the monodromy around the MUM point is conjectured to correspond to the action on $D^b(Y)$ given by tensoring by $\mathcal O(H)$, and the monodromy around the conifold point is conjectured to correspond to the spherical twist by $\mathcal O_Y$.  These two actions on $D^b(Y)$ descend to cohomology as the matrices $$\left(\begin{matrix} 1&0&0&0\\1&1&0&0\\\frac{1}{2}&1&1&0\\ \frac{1}{6}&\frac{1}{2}& 1&1\end{matrix}\right)\text{ and }\left(\begin{matrix} 1&-c&0&-d\\0&1&0&0\\0&0&1&0\\0&0&0&1\end{matrix}\right),$$ respectively, where we've used the basis $1,H,H^2,H^3$ for $H^{ev}(Y,\QQ)$, and $d:=H^3,c:=c_2\cdot H/12$.

In our family this yields the prediction that the mirror $Y$ has $H^3=192,c_2\cdot H=96$.  At first glance this may seem very promising since using Hirzebruch-Riemann-Roch and Kodaira vanishing we may then determine the natural embedding of $Y$ by the linear system $|H|$ and the numerics of the equations that define it.  But as in \cite{ES} we performed a second check of consistency for these predictions by calculating the so-called \textit{conifold period}.  

\subsection{The Conifold Period}The conifold period $f(z)=\int_{C(z)}\Omega$ is defined as the integral of our local generator $\Omega\in \overline{\mathcal F}^3$ against the vanishing cycle $C(z)$ near the conifold point $z=-8$, where $C(z)$ consists of the four $S^3$'s which get collapsed in the formation of the four nodes in the fiber above $z=-8$.  The monodromy around this point is the sum of the four Picard-Lefschetz transformations $$\alpha\mapsto \alpha-\langle\alpha,C_i(z)\rangle C_i(z)$$ associated to each vanishing sphere $C_i(z)$ (See Section 3.2.1 in \cite{V} for a discussion of vanishing spheres and the Picard-Lefschetz transformation).  Since the $C_i(z)$ are all homologous in $H^3(\tilde{X}_z,\QQ)$, this monodromy is a symplectic reflection in $C(z)$ (see \cite{DM}, where it also explained that the number of vanishing spheres should be $|\pi_1(Y)|$).  The Frobenius basis of the Picard-Fuchs equation at the conifold point has three holomorphic solutions and one solution of the form $$f(z)\log z +k(z).$$  Since the monodromy around the conifold point is a symplectic involution in $C(z)$, the conifold period must be $f(z)$ up to scaling.  Analytically continuing this function along a straight line to the MUM point at the origin, we can write this function as $$z_2(s)=\frac{H^3}{6}s^3+\frac{c_2\cdot H}{24}s+\frac{c_3}{(2\pi i)^3}\zeta(3)+O(q),$$ where $q=e^{2\pi i s}$ and $s=\frac{1}{2\pi i} t=\frac{1}{2\pi i} \log z+\frac{1}{2\pi i}g(z)$, in the notation from above (we include a derivation of this form of the conifold period in an appendix for lack of a reference).  Again following the algorithm of \cite{ES}, we numerically calculated the conifold period and scaled so that $H^3=192$ in the above expansion.  This indeed gave $c_2\cdot H=96$, providing an internal consistency check of the above monodromy matrix calculations.  Suprisingly, we also found a positive value for the topological Euler characteristic, $c_3=48$, which cannot correspond to a genuine mirror to our family.  It is again interesting that we nevertheless get integral values for these invariants as one would expect from a genuine geometric situation.  As an alternative to scaling the conifold period to give the expected $H^3$ in the above method, we may also scale to give the necessary Euler characteristic -60, obtaining $H^3=-240,c_2\cdot H=-120$.  Of course, replacing $H$ by $-H$ gives positive values for these.  It is quite possible that there is no way to distinguish between $\pm H$ in this expression of the conifold period.  

\subsection{Monodromy for $\tilde{\mathcal D}$}
As with the calculation of the Yukawa couplings, it seemed worth investigating the behavior near our almost MUM point at infinity.  Repeating the above calculations with the virtual Picard-Fuchs operator $\tilde{\mathcal D}$, one obtains monodromy matrices 
$$
T_0=\left(
\begin{array}{cccc}
1&0&0&0\\
1&1&0&0\\
\frac{1}{2}&1&1&0\\
\frac{1}{6}&\frac{1}{2}&1&1\\
\end{array}
\right),~~
T_{-1/8}=\left(
\begin{array}{cccc}
1&-4&0&-24\\
0&1&0&0\\
0&0&1&0\\
0&0&0&1\\
\end{array}
\right),~~T_{1}=Id,
$$
$$
T_{1/4}=\left(
\begin{array}{cccc}
7&12&24&0\\
1&3&12&24\\
-\frac{5}{2}&-5&-13&-12\\
\frac{7}{6}&\frac{13}{6}&5&3\\
\end{array}
\right),~~
T_{\infty}=\left(
\begin{array}{cccc}
7&16&24&168\\
-8&-17&-36&-168\\
2&4&11&36\\
-\frac{1}{3}&-\frac{2}{3}&-2&-5\\
\end{array}
\right),
$$ giving predictions of $H^3=24,c_2\cdot H=48$.  Calculating the conifold period confirms these numbers and again gives a prediction of $c_3=48$.  We note that like the virtual instanton numbers, these invariants are off by a factor of 4 from those calculated for the operator $\mathcal D$.

\section{Conclusion and Open Questions}
Without an explicit mirror $Y$, we cannot even begin to verify the predictions of mirror symmetry for the family of Calabi-Yau threefolds constructed here.  Nevertheless, we feel that the failed attempts to glean information about the mirror reveal the necessity to understand more deeply the prescriptions for calculating the mirror map and Yukawa couplings used above.  The work here presents some important questions:

1) Does the mirror map need to be altered in some cases?  Would this account for the strange vanishing of odd GV invariants on the conjectural mirror?  Or does torsion in the second homology group of the mirror indeed explain this vanishing? 

2) What is the meaning of the equal, but fractional, roots of the indicial equation for $\mathcal D$ at $\infty$?  Does this have a Hori-Tong GLSM description (see\cite{HoTo}) as suggested by \cite{Kan}?  

3) What is the meaning of the close relationship between invariants computed at 0 and $\infty$?

4) Do we need to reinterpret the correspondence between the vanishing cycle $C\in D\mathcal F(\tilde{X})$ and $\mathcal O_Y\in D^b(Y)$?  Would this explain the strange positive Euler characteristic prediction?  

5) Does the existence of a nontrivial Brauer group for the mirror suggest that derived categories of \textit{twisted} sheaves must be incorporated into the picture of homological mirror symmetry?  Would this explain the strange numerical invariants predicted by homological mirror symmetry?

Hopefully upon answering these questions, we may place the example constructed in this paper in its proper context in mirror symmetry.
\section{Appendix: Expansion of the conifold period}
It was asked in \cite{ES} if the appearance of the term $\frac{c_3}{(2\pi i)^3}\zeta(3)$ in the expansion of the conifold period obtained in \cite{COGP} is a mathematical theorem holding for families of Calabi-Yau threefolds other than the famous example of the quintic threefold treated there.  The purpose of this appendix is to prove that indeed this expansion always holds assuming the validity of the standard conjectures of HMS.  In particular, we assume that the vanishing cycle $S$ in $D\mathcal F(X)$ corresponds to $\mathcal O_Y$ in $D^b(Y)$ for its mirror partner $Y$.

\begin{theorem} Under the above assumptions, the conifold period $z_2(t)$ has the follow expansion up to scaling: $$z_2(t)=\frac{H^3}{6} t^3+\frac{c_2(Y).H}{24}t+\frac{c_3(Y)}{(2\pi i)^3}\zeta(3)+O(q),$$ where of course $H$ is the ample generator of $\Pic(Y)$ and $q=e^{2\pi i t}$ is the usual parameter in the mirror map.
\end{theorem}
\begin{proof}  In order to relate the conifold period to anything on the mirror, we must analytically continue it to the MUM point.  Then according to HMS $S\in D\mathcal F(X)$ should correspond to some object of $D^b(Y)$, $\mathcal O_Y$ by our assumption.  Moreover, according to HMS the pairing $\langle S(z),\Omega\rangle=\int_{S(z)}\Omega$ at the complex moduli point $z$ should correspond to taking the central charge of $\mathcal O_Y$ at the K\"{a}hler moduli point $t$ corresponding to $z$ (see comments on pages 2 and 10 of \cite{Iri}).  Of course the coordinate $t$ corresponds to the complexified divisor $tH$, where $H$ is the ample generator of $\Pic(Y)$.  

From the fourth formula on page 12 of \cite{Iri}, we get that this central charge is 

$$Z_{tH}(\mathcal O_Y)=-\frac{c_3(Y)}{(2 \pi i)^3}\zeta(3) -\frac{c_2(Y).H}{24}\frac{t}{2 \pi i} +\frac{1}{(2 \pi i)^3} G(tH).$$ Here the function $G$ is defined by 

$$G(tH):=2F_0(tH)-t\frac{d}{dt} F_0(tH),\text{ where }F_0(tH):=\frac{H^3}{6}t^3+\sum_{d>0}N_d e^{d H.C t}$$ is the genus zero potential of $Y$ and $C$ is the generator of $H_2(Y,Z)$ modulo torsion.  Expanding out $G(tH)$, we get

$$G(tH)= -\frac{H^3}{6} t^3-t (\sum_{d>0} dH.C N_d e^{d H.C t})+2(\sum_{d>0} N_d e^{d H.C t}).$$  Putting this back into the formula above for the central charge and noticing that the terms different from $-\frac{H^3}{6}t^3$ can be grouped into the $O(q)$ term, we get the desired form of the conifold period up to scaling by $-1$.
\end{proof}

\end{document}